\pdfoutput=1
\RequirePackage{ifpdf}
\ifpdf 
\documentclass[pdftex]{sigma}
\else
\documentclass{sigma}
\fi

\numberwithin{equation}{section}

\newtheorem{Theorem}{Theorem}[section]
\newtheorem{Corollary}[Theorem]{Corollary}
\newtheorem{Lemma}[Theorem]{Lemma}
\newtheorem{Proposition}[Theorem]{Proposition}
 { \theoremstyle{definition}
\newtheorem{Definition}[Theorem]{Definition}
\newtheorem{Example}[Theorem]{Example}
\newtheorem{Remark}[Theorem]{Remark} }

\usepackage[all]{xy}

\usepackage{tabularx}
\usepackage{diagbox}

\def\Hom{\mathop{\rm Hom}\nolimits}
\def\Ext{\mathop{\rm Ext}\nolimits}

\def\mod{\mathop{\rm mod}\nolimits}
\def\rad{\mathop{\rm rad}\nolimits}

\def\tilt{\mathop{\rm \tau\makebox{-}tilt}\nolimits}
\def\stilt{\mathop{\rm s\tau\makebox{-}tilt}\nolimits}
\def\pstilt{\mathop{\rm ps\tau\makebox{-}tilt}\nolimits}

\newcommand{\zL}{\Lambda}

\begin{document}
\allowdisplaybreaks

\newcommand{\arXivNumber}{2002.02990}

\renewcommand{\thefootnote}{}

\renewcommand{\PaperNumber}{058}

\FirstPageHeading

\ShortArticleName{On the Number of $\tau$-Tilting Modules over Nakayama Algebras}

\ArticleName{On the Number of $\boldsymbol{\tau}$-Tilting Modules\\ over Nakayama Algebras\footnote{This paper is a~contribution to the Special Issue on Cluster Algebras. The full collection is available at \href{https://www.emis.de/journals/SIGMA/cluster-algebras.html}{https://www.emis.de/journals/SIGMA/cluster-algebras.html}}}

\Author{Hanpeng GAO~$^\dag$ and Ralf SCHIFFLER~$^\ddag$}

\AuthorNameForHeading{H.~Gao and R.~Schiffler}

\Address{$^\dag$~Department of Mathematics, Nanjing University, Nanjing 210093, P.R.~China}
\EmailD{\href{mailto:hpgao07@163.com}{hpgao07@163.com}}

\Address{$^\ddag$~Department of Mathematics, University of Connecticut, Storrs, CT 06269-1009, USA}
\EmailD{\href{mailto:schiffler@math.uconn.edu}{schiffler@math.uconn.edu}}

\ArticleDates{Received March 06, 2020, in final form June 11, 2020; Published online June 18, 2020}

\Abstract{Let $\Lambda^r_n$ be the path algebra of the linearly oriented quiver of type $\mathbb{A}$ with~$n$ vertices modulo the $r$-th power of the radical, and let $\widetilde{\Lambda}^r_n$ be the path algebra of the cyclically oriented quiver of type $\widetilde{\mathbb{A}}$ with $n$ vertices modulo the $r$-th power of the radical. Adachi gave a recurrence relation for the number of $\tau$-tilting modules over $\Lambda^r_n$. In this paper, we show that the same recurrence relation also holds for the number of $\tau$-tilting modules over $\widetilde{\Lambda}^r_n$. As an application, we give a new proof for a result by Asai on recurrence formulae for the number of support $\tau$-tilting modules over $\Lambda^r_n$ and $\widetilde{\Lambda}^r_n$.}


\Classification{16G20; 16G60}

\section{Introduction}

The starting point of {\em tilting theory} was the introduction of tilting modules over a hereditary algebra by Happel and Ringel in \cite{HR}. Ever since, the study of tilting modules and their endomorphism algebras has been an important branch of representation theory.

About 25 years later, {\em cluster-tilting theory}, introduced in \cite{BMRRT} (and \cite{CCS} in type $\mathbb{A}$), brought a~new perspective to the subject by replacing the hereditary algebra by its cluster category, and a~direct relation between tilted and cluster-tilted algebras was then established in~\cite{ABS}.

In 2014, Adachi, Iyama and Reiten \cite{AIR} introduced {\em $\tau$-tilting theory} replacing the rigidity condition $\Ext^1_{\Lambda}(T,T)=0$ for a~tilting module by the weaker condition $\Hom_\Lambda(T,\tau_\Lambda T)=0$ for a~$\tau$-tilting module, where $\tau$ denotes the Auslander--Reiten translation, and $\Lambda$ is any finite-dimensional algebra. They showed that, in contrast to tilting modules but in agreement with cluster-tilting objects, it is always possible to exchange a~given indecomposable summand of a~support $\tau$-tilting module for a~unique other indecomposable and obtain a new support $\tau$-tilting module. This process, called mutation, is essential in cluster theory.

In the same paper, the authors also showed that the support $\tau$-tilting modules are in bijection with several other important classes in representation theory including functorially finite torsion classes introduced in~\cite{AS1981}, 2-term silting complexes introduced in \cite{KV1988}, and cluster-tilting objects in the cluster category when the algebra $\Lambda$ is hereditary, or, more generally, cluster-tilted.

Therefore, it is natural to ask what is the number of support $\tau$-tilting modules over a given algebra.

For a hereditary algebra, the support $\tau$-tilting modules are exactly the support tilting mo\-dules. Moreover they are in bijection with the cluster-tilting objects in the cluster category and hence with the clusters in the corresponding cluster algebra. For Dynkin type, these numbers were first calculated in \cite{FZYsyst} via cluster algebras, and later in \cite{ONFR} via representation theory. In particular, over a hereditary algebra of type $\mathbb{A}_n$, the number of tilting modules is~$C_n$, and the number of support tilting modules is $C_{n+1}$ where $C_i$ denotes the $i$-th Catalan number $\frac{1}{i+1}\binom{2i}{i}$.

In this paper, we study this question over finite-dimensional Nakayama algebras. Recall that a finite-dimensional $K$-algebra is said to be a {\em Nakayama algebra} if every indecomposable projective module and every indecomposable injective module has a unique composition series. Nakayama algebras come in two types, in fact, a finite-dimensional algebra is Nakayama if and only if its quiver is one of the following
\begin{gather*}\xymatrix@!@R=2pt@C=2pt{
A_n\colon &1\ar[r]&2\ar[r]&3\ar[r]&\cdots\ar[r]&n,&&\widetilde{A}_n\colon &1\ar[r]&2\ar[r]&3\ar[r]&\cdots\ar[r]&n,\ar@/_20pt/[llll]
}\end{gather*}
see \cite[Section~V.3.2]{ASS2006}. Throughout the paper, we use the following notation
\begin{gather*} \Lambda_n^r=KA_n/\rad^r \qquad \text{and}\qquad \widetilde{\Lambda}_n^r=K\widetilde{A}_n/\rad^r.\end{gather*}
Moreover, we let $t_r(n)$ and $\tilde t_r(n)$ denote the number of $\tau$-tilting modules over $\Lambda^r_n$ and $\widetilde{\Lambda}^r_n$, and let~$s_r(n)$ and~$\tilde s_r(n)$ denote the number of support $\tau$-tilting modules over $\Lambda^r_n$ and $\widetilde{\Lambda}^r_n$, respectively. We also set $t_r (n) := 0$ and $s_r (n) := 0$ for $n < 0$.

Adachi classified $\tau$-tilting modules over Nakayama algebras in~\cite{Ada2016}. Under the assumption that the Loewy length of every indecomposable projective module is at least~$n$, Adachi showed that the number of $\tau$-tilting modules is exactly $\binom{2n-1}{n-1}$ and the number of support $\tau$-tilting modules is $\binom{2n}{n}$. Moreover, he also gave the following recurrence relation for the number $t_r(n)$ of $\tau$-tilting modules over $\Lambda^r_n$
\begin{gather*}t_r(n)=\sum^r_{i=1}C_{i-1}\cdot t_r(n-i).\end{gather*}

The aim of this paper is study the number of $\tau$-tilting modules over $\widetilde{\Lambda}^r_n$. We show that there is a close relationship between the number $\tilde t_r(n)$ of $\tau$-tilting $\widetilde{\Lambda}_n^r$-modules and the number $t_r(n)$ of $\tau$-tilting $\Lambda_n^r$-modules.

\begin{Proposition}[see~Proposition \ref{3.7}]
\begin{gather*}\tilde t_r(n)=\sum^r_{i=1}i\cdot C_{i-1}\cdot t_r(n-i).\end{gather*}
\end{Proposition}

Next, we prove that the functions $t$ and $ \tilde t$ satisfy the same recurrence relation.

\begin{Theorem}[see Theorem~\ref{3.8}]\label{thm A}
We have the following recurrence relation
\begin{gather*}\tilde t_r(n)=\sum\limits^r_{i=1}C_{i-1}\cdot \tilde t_r(n-i). \end{gather*}
\end{Theorem}

As an application, we obtain a new proof for the following result by Asai on the numbers of support $\tau$-tilting modules $s_r(n)$ and $\tilde s_r(n).$
\begin{Theorem}[{\cite[Theorem 4.1]{S2018}}]\label{thm B}\quad
\begin{enumerate}\itemsep=0pt
\item[$(1)$]$s_r(n)=2s_r(n-1)+\sum\limits^r_{i=2}C_{i-1}\cdot s_r(n-i).$
\item[$(2)$]$\tilde s_r(n)=2\tilde s_r(n-1)+\sum\limits^r_{i=2}C_{i-1}\cdot \tilde s_r(n-i).$
\end{enumerate}
\end{Theorem}
Asai used a bijection between support $\tau$-tilting modules and semibricks to obtain his result in the context of semibricks. Our proof is combinatorial.

The paper is organized as follows. In Section~\ref{sect 2}, we fix the notation and recall several results of Adachi that are relevant to this paper. We study the Nakayama algebras of type ${A}_n$ and $\widetilde{A}_n$ and prove Theorem~\ref{thm A} in Section~\ref{sect 3}. Theorem~\ref{thm B} is proved in Section~\ref{sect 4}. We include tables of the numbers of $\tau$-tilting and support $\tau$-tilting modules in Section~\ref{sect 5}.

\section{Preliminaries}\label{sect 2}

Throughout this paper, all algebras will be basic, connected, finite-dimensional algebras over an algebraically closed field $K$
 and all modules will be finitely generated right modules. For an algebra $\Lambda$, we denote by $\mod \Lambda$ the category of finitely generated right $\Lambda$-modules and by $\tau_{\Lambda}$ the Auslander--Reiten translation of $\Lambda$. Let $\{e_1, e_2, \dots, e_n\}$ be a complete set of primitive orthogonal idempotents of $\Lambda$. We put $P_i=e_i\Lambda$ the indecomposable projective module and $S_i=P_i/\operatorname{rad}P_i$ the simple module of~$\Lambda$ for $i=1,2,\dots,n$. For $M\in \mod \Lambda$, we denote by $\ell(M)$ the Loewy length of $M$ and by $|M|$ the number of pairwise nonisomorphic indecomposable summands of~$M$. For a finite set~$X$, we denote by~$|X|$ the cardinality of $X$.
For details on representation theory of finite-dimensional algebras we refer to~\cite{ASS2006,S}.

Let $\Lambda$ be an algebra. In this section, we recall results about support $\tau$-tilting modules that are needed later.

\begin{Definition}\label{2.1}Let $M\in\mod\Lambda$.
\begin{enumerate}\itemsep=0pt
\item[(1)] $M$ is called {\em $\tau$-rigid} if $\Hom_\Lambda(M,\tau_\Lambda M)=0$.
\item[(2)] $M$ is called {\em $\tau$-tilting} if it is $\tau$-rigid and $|M|=|\Lambda|$.
\item[(3)] $M$ is called {\em support $\tau$-tilting} if it is a $\tau$-tilting $\Lambda/\Lambda e\Lambda$-module for some idempotent $e$ of $\Lambda$.
\item[(4)] $M$ is called {\em proper support $\tau$-tilting} if it is a support $\tau$-tilting but not a $\tau$-tilting $\Lambda$-module.
\end{enumerate}
\end{Definition}

Recall that $M\in\mod \Lambda$ is called {\it sincere} if every simple $\Lambda$-module appears as a composition factor in $M$. It is well-known that the $\tau$-tilting modules are exactly the sincere support $\tau$-tilting modules \cite[Proposition 2.2(a)]{AIR}.

We will denote by $\tilt\Lambda$ (respectively, $\stilt\Lambda$, $\pstilt\Lambda$) the set of isomorphism classes of basic $\tau$-tilting (respectively, support $\tau$-tilting, proper support $\tau$-tilting) $\Lambda$-modules. Obviously, we have $|\stilt\Lambda|=|\tilt\Lambda|+|\pstilt\Lambda|$.

Let $\pstilt_{np}\Lambda:=\{M\in \pstilt\Lambda \,|\, M \ \text{has no projective direct summands}\}$. We recall the following results proved by Adachi in \cite{Ada2016}.

\begin{Theorem}[{\cite[Theorem 2.6]{Ada2016}}]\label{2.2}
Let $\Lambda$ be a Nakayama algebra. There is a bijection between $\tilt\Lambda$ and $\pstilt_{np}\Lambda$.
\end{Theorem}

The following result is very useful.

\begin{Proposition}[{\cite[Proposition 2.32]{Ada2016}}]\label{2.3} Let $\Lambda$ be a Nakayama algebra of type~$A_n$. Then each $\tau$-tilting $\Lambda$-module has~$P_1$ as a direct summand.
\end{Proposition}

 The following recurrence relations are very useful to calculate the number of $\tau$-tilting modules over Nakayama algebras of type $A_n$.

\begin{Lemma}[{\cite[Corollary 2.34]{Ada2016}}]\label{2.4} Let $\Lambda$ be a Nakayama algebra of type $A_n$. Then
 \begin{gather*}|\tilt\Lambda|=\sum^{\ell(P_{1})}_{i=1}C_{i-1}\cdot |\tilt(\Lambda/\langle e_{\leqslant i}\rangle )|,\end{gather*}
 where $e_{ \leqslant i}:=e_1+e_2+\dots+e_i$. In particular, for $\Lambda=\Lambda_n^r$, we have
\begin{gather*}t_r(n)=\sum^r_{i=1}C_{i-1}\cdot t_r(n-i),\end{gather*}
where $t_r(n)$ is the number of $\tau$-tilting modules over~$\Lambda^r_n$.
\end{Lemma}

See Table~\ref{t1} in Section~\ref{sect 5} for explicit values of $t_r(n)$.

\begin{Remark}\label{2.5}
If $r\ge n$ then $\Lambda^r_n$ is hereditary, and we have $t_r(n)=C_n.$ On the other hand, Lemma~\ref{2.4} yields the equation $t_n(n)=\sum\limits^n_{i=1}C_{i-1}\cdot t_n(n-i)$. Hence we recover the well-known combinatorial identity
\begin{gather*}C_n=\sum^n_{i=1}C_{i-1}\cdot C_{n-i}.\end{gather*}
\end{Remark}

\section[The number of $\tau$-tilting modules over Nakayama algebras]{The number of $\boldsymbol{\tau}$-tilting modules over Nakayama algebras}\label{sect 3}

In this section, $\Lambda$ will be any Nakayama algebra of type $A_n$. As usual, we use the notation~$t_r(n)$ for the number of $\tau$-tiling modules over $\Lambda_n^r=KA_n/ {\rm rad}^r$ and $\widetilde{t}_r(n)$ for the number of $\tau$-tiling modules over $\widetilde{\Lambda}_n^r=K\widetilde{A}_n/ {\rm rad}^r$. We will show that the functions~$t $ and~$ \tilde t$ satisfy the same recurrence relation.

We denote by $W_i$ $(i=1,2,\dots,n)$ the set of support $\tau$-tilting $\Lambda$-modules which have the simples $S_1,S_2,\dots, S_{i-1}$ as composition factor but not $S_i$. We also write
$\Lambda_{>i}:=\Lambda/\langle e_{\leqslant i}\rangle$ and $\Lambda_{<i}:=\Lambda/\langle e_{\geqslant i}\rangle$ where $e_{ \leqslant i}:=e_1+e_2+\cdots+e_i$ and $e_{ \geqslant i}:=e_i+e_{i+1}+\cdots+e_n$.
\begin{Lemma}\label{3.1}$|W_i|=|\tilt\Lambda_{<i}|\cdot|\stilt\Lambda_{> i}|$.
\end{Lemma}
\begin{proof} Since the quiver of $\Lambda$ is tree, we have $\Lambda/\langle e_i\rangle\cong \Lambda_{<i}\times \Lambda_{>i}$. Thus there is a bijection
\begin{gather*}\tilt \Lambda_{<i}\times \stilt\Lambda_{> i}\longrightarrow W_i\end{gather*}
given by $(N_1,N_2)\mapsto N_1\oplus N_2$ where $N_1$ is a $\tau$-tilting $\Lambda_{<i}$-module and $N_2$ is a support $\tau$-tilting $\Lambda_{>i}$-module. Hence $|W_i|=|\tilt\Lambda_{<i}|\cdot|\stilt\Lambda_{>i}|$.
\end{proof}

\begin{Proposition}\label{3.2} Let $\Lambda$ be a Nakayama algebra of type $A_n$. We have
\begin{enumerate}\itemsep=0pt
\item[$(1)$] $|\pstilt\Lambda|=\sum\limits^{n}_{i=1}|\tilt\Lambda_{< i}|\cdot|\stilt\Lambda_{>i}|$,
\item[$(2)$] $|\stilt\Lambda|=\sum\limits^{n}_{i=1}|\tilt\Lambda_{<i} |\cdot|\stilt\Lambda_{>i}|+\sum\limits^{\ell(P_{1})}_{i=1}C_{i-1}\cdot |\tilt\Lambda_{>i}|$,
\item[$(3)$] $|\pstilt\Lambda|=\sum\limits^{n}_{i=1}|\stilt\Lambda_{<i}|\cdot|\tilt\Lambda_{>i}|$,
\item[$(4)$] $|\stilt\Lambda|=\sum\limits^{n}_{i=1}|\stilt\Lambda_{<i}| \cdot|\tilt\Lambda_{>i}|+\sum\limits^{\ell(P_{1})}_{i=1}C_{i-1}\cdot |\tilt\Lambda_{>i}|$.
\end{enumerate}
\end{Proposition}
\begin{proof}(1)
Since $\pstilt\Lambda=\bigcup\limits^n_{i=1}W_i$, we have
\begin{gather*}|\pstilt\Lambda|=\sum\limits^{n}_{i=1}|W_i| =\sum\limits^{n}_{i=1}|\tilt\Lambda_{<i}|\cdot|\stilt\Lambda_{>i}|
\end{gather*} by Lemma~\ref{3.1}.

(2) Since $|\stilt\Lambda|=|\pstilt\Lambda|+|\tilt\Lambda|$, the statement follows from Lemma \ref{2.4}.

(3) is similar to (1) and
(4) follows from (3).
\end{proof}

We give an example of Proposition~\ref{3.2}.
\begin{Example}\label{3.3} Let $\Lambda$ be an algebra is given by the quiver $\xymatrix@C10pt{1\ar[r]^\alpha&2\ar[r]^\beta&3\ar[r]&4}$  with the relation $\alpha\beta=0$.

 \begin{table}[h!]\small\centering
\begin{tabular}{c|c|c|c|c|c|c}
 \hline
&$\Lambda_{<i}$ & $\Lambda_{>i}$& $|\tilt\Lambda_{<i}|$&$|\stilt\Lambda_{>i}|$&$|\stilt\Lambda_{<i}|$&$|\tilt\Lambda_{>i}|$\\
 \hline
$i=1$&0 & $\xymatrix@C10pt{2\ar[r]^\beta&3\ar[r]&4}$&1&14&1&5\\
$i=2$& 1& $\xymatrix@C10pt{3\ar[r]&4}$ & 1 &5&2&2 \\
$i=3$& $\xymatrix@C10pt{1\ar[r]^\alpha&2}$ &4 &2 &2&5&1\\
$i=4$& $\xymatrix@C10pt{1\ar[r]^\alpha&2\ar[r]^\beta&3}$, $\alpha\beta=0$ & 0& 3 & 1&12&1\\
 \hline \end{tabular}
\end{table}

By Proposition \ref{3.2}(1), $|\pstilt\Lambda|=1\cdot 14+1\cdot 5+2\cdot 2+3\cdot 1=26$.
 Note that $\ell(P_1)=2$, and thus Lemma~\ref{2.4} implies $|\tilt\Lambda|=1\cdot 5+1\cdot 2=7$. Hence, we have $|\stilt\Lambda|=26+7=33$ by Proposition \ref{3.2}(2). Moreover, we can also use part (3) of Proposition~\ref{3.2} and compute $|\pstilt\Lambda|=1\cdot 5+2\cdot2+5\cdot 1+12\cdot1= 26$.
\end{Example}

\begin{Corollary}\label{3.4} Let $V_\ell~(\ell=1,2,\dots,n)$ be the set of all support $\tau$-tilting $\Lambda$-modules which have $S_\ell,S_{\ell-1},\dots,S_1$ as composition factor. Then we have
\begin{gather*}|V_\ell |=\sum\limits^{n}_{i=\ell+1}|\tilt\Lambda_{<i}|\cdot|\stilt\Lambda_{>i}|+|\tilt\Lambda|.\end{gather*}
\end{Corollary}
\begin{proof}
This result follows from $V_\ell=\Big(\bigcup\limits^n_{i=\ell+1}W_i\Big)\bigcup\tilt\Lambda$.
\end{proof}

From now, we will study the number of $\tau$-tilting $\widetilde{\Lambda}^r_n$-modules.
The following result is very useful to calculate the number of proper support $\tau$-tilting $\widetilde{\Lambda}^r_n$-modules.
\begin{Proposition}\label{3.5}
\[
\big|\pstilt\widetilde{\Lambda}^r_n\big|=\sum\limits^{n-1}_{i=1}i\cdot t_r(i-1)\cdot s_r({n-i-1})+n\cdot t_r(n-1).
\]
\end{Proposition}
\begin{proof}
For $1\leqslant \ell\leqslant n-1$, we define
 \begin{gather*}K_{n,\ell}=\big\{N\in \stilt\widetilde{\Lambda}^r_n \,|\, N~\text{has $S_n,S_{n-1},\dots, S_{n-\ell+1}$ as composition factor but not $S_{n-\ell}$}\big\}\end{gather*}
 and
 \begin{gather*}K_{n,0}=\big\{N\in \stilt\widetilde{\Lambda}^r_n \,|\, N~\text{does not contain $S_n$ as composition factor}\big\}.\end{gather*}
 Note that $\widetilde{\Lambda}^r_n/\langle e_{n-\ell}\rangle$ is the quotient of path algebra of the quiver
 \begin{gather*}\xymatrix{n-\ell+1\ar[r]&\cdots\ar [r] &n \ar[r]&1\ar[r]& 2\ar[r]&\cdots\ar[r]& n-\ell-1}\end{gather*}
 by the 2-sided ideal generated by paths of length $r$. Therefore
 $\widetilde{\Lambda}^r_n/\langle e_{n-\ell}\rangle\cong \Lambda^r_{n-1}$.
By Corollary~\ref{3.4}, we have
 \begin{gather*}|K_{n,\ell}|=|V_{\ell}|=\sum\limits^{n-1}_{i=\ell+1}t_r({i-1})\cdot s_r({n-i -1})+t_r({n-1}).\end{gather*}
In particular, $|K_{n,0}|=s_r({n-1})=\sum\limits^{n-1}_{i=1} t_r({i-1})\cdot s_r({n-i -1})+t_r({n-1}).$
Hence,
\begin{align*}
\big|\pstilt\widetilde{\Lambda}^r_n\big| &= \sum\limits^{n-1}_{\ell=0}|K_{n,\ell}|
=\sum\limits^{n-1}_{\ell=0}(\sum\limits^{n-1}_{i=\ell+1} t_r(i-1)\cdot s_r(n-i-1)+t_r(n-1))\\
&=\sum\limits^{n-1}_{i=1}i\cdot t_r(i-1)\cdot s_r(n-i-1)+n\cdot t_r(n-1).\tag*{\qed}
\end{align*}\renewcommand{\qed}{}\end{proof}

 We set $X_n=\varnothing $, for $n<0$, and for all $n\ge 0$, we define
 \begin{gather*}X_{n}=\big\{N\in \stilt\Lambda^r_n\,|\, N \ \text{does not contain $P_{n-r+1},P_{n-r},\dots,P_1$ as direct summands} \big\}\end{gather*}
and \begin{gather*}Y_{n,\ell}=\{N\in X_n \,|\, N \ \text{contains $S_{\ell},S_{\ell-1},\dots,S_1$ as composition factor} \}, \quad \ell=1,2,\dots,n.\end{gather*}
 In particular, $X_{n}= \stilt\zL^r_n$ when $n+1\leqslant r$, and $Y_{n,0}=X_n$.
\begin{Lemma}\label{3.6} With the above notions, we have
\begin{enumerate}\itemsep=0pt
\item[$(1)$] $|X_{n}|=t_r(n+1)$,
\item[$(2)$] $|Y_{n,\ell}|=
\begin{cases}
\sum\limits^{r}_{i=\ell+1}C_{i-1}\cdot t_r({n-i+1})& \text{if $\ell\leqslant r-1$},\\
0&\text{if $\ell\geqslant r$}.
\end{cases}$
\end{enumerate}
\end{Lemma}
\begin{proof}(1) Write $\Lambda=\zL^r_n$, we have
$\Lambda/\langle e_i\rangle\cong \Lambda_{< i}\times \Lambda_{> i}$ for a given $i$.
Let $Z_1$ be the set of all support $\tau$-tilting $\Lambda_{<i}$-modules which have $S_1,S_2,\dots,S_{i-1}$ as composition factor (they are exactly $\tau$-tilting $\Lambda_{< i}$-modules)
and they don't have $P_1,P_2,\dots,P_{n-r+1}$ as direct summands.

If $i\leqslant r$, then $Z_1=\tilt\Lambda_{< i}$.

If $i\geqslant r+1$, then all $\tau$-tilting $\Lambda_{<i}$-modules have $P_1$ as direct summand by Proposition~\ref{2.3}, and hence $Z_1=\varnothing$.

Let $Z_2$ be the set of all support $\tau$-tilting $\Lambda_{>i}$-modules who do not have $P_1,P_2,\dots,P_{n-r+1}$ as direct summands. Then $Z_2$ consists of exactly all support $\tau$-tilting $\Lambda_{> i}$\,-modules which do not contain
$P_{i+1},P_{i+2},\dots,P_{n-r+1}$ as direct summands.
Denoted by $X_{n,i}\subseteq X_n$ $(i=1,2,\dots,n)$ the subset of all support $\tau$-tilting modules having $S_1,S_2,\dots, S_{i-1}$ as composition factor but not~$S_i$ and by $X_{n,n+1}\subseteq X_n$ the subset of all support $\tau$-tilting modules having $S_1,S_2,\dots, S_{n}$ as composition factor (hence, they are exactly $\tau$-tilting). We have $|Z_2|=|X_{n-i}|$, since $\Lambda_{> i}\cong \zL^r_{n-i}$.
There is a bijection between $Z_1\times Z_2$ and $X_{n,i}$ given
by $(N_1, N_2)\to N_1\oplus N_2$ where $N_1\in Z_1$ and $N_2\in Z_2$.
Therefore, $|X_{n,i}|=|Z_1|\cdot |Z_2|$.
Thus, we have
\begin{gather*}|X_{n,i}|=
\begin{cases}
|\tilt\Lambda_{< i}|\cdot |X_{n-i}|& \text{if} \ i\leqslant r,\\
0&\text{if} \ i\geqslant r+1.
\end{cases}\end{gather*}
 If $r\leqslant n$, we have
\begin{gather*}|X_n|=\sum^{n+1}_{i=1}|X_{n,i}|=\sum^{r}_{i=1}|\tilt\Lambda_{< i}|\cdot |X_{n-i}| =\sum^{r}_{i=1}C_{i-1}\cdot |X_{n-i}|.\end{gather*}
Note that, if $r\ge n+1$ then $X_n= \stilt\zL^r_n$ and hence
$|X_n|=C_{n+1}$ since $\Lambda^r_n$ is hereditary, we get $|X_n|=t_r(n+1)$.
On the other hand, Lemma \ref{2.4} says that $t_r(n)$ and $|X_n|$ satisfy the same recursive formula. Thus
we have $|X_n|=t_r(n+1)$ for all $n$, by induction.

(2) follows from (1) and the fact $Y_{n,\ell}=\bigcup\limits^{n+1}_{i=\ell+1}X_{n,i}$.
\end{proof}

As a result of Lemma \ref{3.6}, we have the following relationship between $\tilde t_r(n)$ and $t_r(n)$.

\begin{Proposition} \label{3.7}We have
\begin{gather*}\tilde t_r(n)=\sum^r_{i=1}i\cdot C_{i-1}\cdot t_r({n-i}).\end{gather*}
\end{Proposition}
\begin{proof}Considering the set $K^{np}_{n,\ell}$ consisting of all modules in $K_{n,\ell}$ which do not have projective $\widetilde{\Lambda}^r_n$-modules as direct summands.

Note that the indecomposable projective $\widetilde{\Lambda}^r_n/\langle e_{n-\ell}\rangle$-modules corresponding to the last $r-1$ points (the length of them is at most $r-1$) are not projective $\widetilde{\Lambda}^r_n$-modules, we have
$\big|K^{np}_{n,\ell}\big|=|Y_{n-1,\ell}|$ since $\widetilde{\Lambda}^r_n/\langle e_{n-\ell}\rangle\cong \zL^r_{n-1}$.
Thus,
\begin{align*}
\big|\pstilt_{np}\widetilde{\Lambda}^r_n\big|
&= \sum\limits^{n-1}_{\ell=0}\big|K^{np}_{n,\ell}\big| =\sum\limits^{n-1}_{\ell=0}|Y_{n-1,\ell}|\\
&=\sum\limits^{r-1}_{\ell=0}\sum\limits^{r}_{i=l+1}C_{i-1}\cdot t_r(n-i)\quad \text{(by~Lemma~\ref{3.6})}\\
&=\sum^r_{i=1}i\cdot C_{i-1}\cdot t_r(n-i).
\end{align*}
Therefore, the assertion follows from Theorem \ref{2.2}.
\end{proof}

Now, we are ready to prove our main result of this section.

\begin{Theorem}\label{3.8}
We have
\begin{gather*} \tilde t_r(n)=\sum\limits^r_{i=1}C_{i-1}\cdot \tilde t_r(n-i). \end{gather*}
\end{Theorem}

\begin{proof}By Proposition \ref{3.7}, we have $\tilde t_r(n)=\sum\limits^r_{\ell=1}\ell\cdot C_{\ell-1}\cdot t_r(n-\ell)$.
Thus,
\begin{gather*}
\tilde t_r(n)-\sum\limits^r_{i=1}C_{i-1}\cdot\tilde t_r(n-i)
=\sum\limits^r_{\ell=1}\ell\cdot C_{\ell-1}\cdot t_r(n-\ell)\\
\hphantom{\tilde t_r(n)-\sum\limits^r_{i=1}C_{i-1}\cdot\tilde t_r(n-i)=}{}
-\sum\limits^r_{i=1}C_{i-1}\cdot \left(\sum\limits^r_{\ell=1}\ell\cdot C_{\ell-1}\cdot t_r(n-i-\ell)\right)\\
\hphantom{\tilde t_r(n)-\sum\limits^r_{i=1}C_{i-1}\cdot\tilde t_r(n-i)}{}
=\sum\limits^r_{\ell=1}\ell\cdot C_{\ell-1}\cdot \left(t_r(n-\ell)-\sum\limits^r_{i=1}C_{i-1}\cdot t_r(n-\ell-i)\right)\\
\hphantom{\tilde t_r(n)-\sum\limits^r_{i=1}C_{i-1}\cdot\tilde t_r(n-i)}{} =0. \quad \text{(by~Lemma~\ref{2.4})}
\end{gather*}
Hence, $\tilde t_r(n)=\sum\limits^r_{i=1}C_{i-1}\cdot\tilde t_r(n-i)$.\end{proof}

The following proposition and its proof are similar to \cite[Theorem 4.1 (3)]{S2018}. For convenience, we include a proof here. We shall use the following notation.

For every positive integer $r$, let $F_r(X)=\sum\limits^r_{i=0}c_{i}\cdot X^{r-i}$ where $c_0=1$ and $c_i=-C_{i-1}$ for $i=1,2,\dots, r$. Let
\begin{gather*}{\bf E}_n(X_1,X_2,\dots,X_r)=\sum_{J\subseteq\{1,2,\dots,r\}, |J|=n}\prod_{j\in J}X_j \end{gather*} be the $n$-th elementary symmetric polynomial, $n=0,1,2,\dots, r.$ Let
\begin{gather*}{\bf H}_n(X_1,X_2,\dots,X_r)=\sum\limits_{\substack{t_1,t_2,\dots,t_r\in \mathbb{Z}_{\geqslant0}\\ t_1+t_2+\dots+t_r=n}}X^{t_1}_1X^{t_2}_2\cdots X^{t_r}_r
\qquad \text{for all }n\in \mathbb{Z},\\
{\bf P}_n(X_1,X_2,\dots,X_r)=\sum\limits^r_{i=1}X_i^n \quad \text{for all }n\geqslant 1.\end{gather*}
In particular, we have ${\bf E}_0=1$, ${\bf H}_0=1$, and ${\bf H}_n=0 $ for $n< 0$.

\begin{Proposition}\label{3.9}Let $\xi_1,\xi_2,\dots,\xi_r$ be the roots $($not necessarily distinct$)$ of the polyno\-mial $F_r(X)$.
Then we have
\begin{enumerate}\itemsep=0pt
\item[$(1)$] ${t}_r(n)=\sum\limits_{\substack{t_1,t_2,\dots,t_r\in \mathbb{Z}_{\geqslant0}\\ t_1+t_2+\dots+t_r=n}}\xi^{t_1}_1\xi^{t_2}_2\cdots \xi^{t_r}_r$,
\item[$(2)$] $\widetilde{t}_r(n)=\sum\limits^r_{i=1}\xi^n_i$.
\end{enumerate}
\end{Proposition}

\begin{proof} Using Vieta's formula on symmetric polynomials, we have \begin{gather*}{\bf E}_i(\xi_1,\xi_2,\dots,\xi_r)=(-1)^ic_i\end{gather*} for $i=0,1,2,\dots,r$. By \cite[Lemma 4.8]{S2018}, we have
\begin{gather*}\sum\limits_{i=0}^rc_i{\bf H}_{n-i}(\xi_1,\xi_2,\dots,\xi_r)=0.\end{gather*}
 On the other hand, Lemma \ref{2.4} yields $t_r(n) =\sum\limits^r_{i=1}C_{i-1}\cdot t_r(n-i)$ which implies
 \begin{gather*}\sum^r_{i=0}c_{i}\cdot t_r(n-i)=0.\end{gather*} Therefore
 \begin{gather}\label{eq sym}
\sum^r_{i=0}c_{i}\cdot (t_r(n-i)-{\bf H}_{n-i}(\xi_1,\xi_2,\dots,\xi_r))=0.
\end{gather}
Note that $t_r(0)=1={\bf H}_{0}(\xi_1,\xi_2,\dots,\xi_r)$ and $t_r(n)=0={\bf H}_{n}(\xi_1,\xi_2,\dots,\xi_r)$ for $n< 0$.
Therefore, using induction and equation (\ref{eq sym}), we see that for all $n$ \begin{gather*}t_r(n)={\bf H}_{n}(\xi_1,\xi_2,\dots,\xi_r)=\sum\limits_{\substack{t_1,t_2,\dots,t_r\in \mathbb{Z}_{\geqslant0}\\ t_1+t_2+\cdots+t_r=n}}\xi^{t_1}_1\xi^{t_2}_2\cdots \xi^{t_r}_r.\end{gather*}
In Proposition \ref{3.7}, we have shown the following relation
\begin{gather*}\tilde t_r(n)=\sum^r_{i=1}-i\cdot c_i\cdot t_r({n-i}).\end{gather*}
Hence,\begin{align*}
\tilde t_r(n)&=\sum^r_{i=1}-i\cdot c_{i}\cdot t_r({n-i})
 =\sum^r_{i=1}-i\cdot c_{i}\cdot {\bf H}_{n-i}(\xi_1,\xi_2,\dots,\xi_r)\\
&=\sum^r_{i=1}-i\cdot (-1)^i\cdot {\bf E}_i(\xi_1,\xi_2,\dots,\xi_r)\cdot {\bf H}_{n-i}(\xi_1,\xi_2,\dots,\xi_r)\\
&=\sum^r_{i=1}(-1)^{i-1}\cdot i\cdot {\bf E}_i(\xi_1,\xi_2,\dots,\xi_r)\cdot {\bf H}_{n-i}(\xi_1,\xi_2,\dots,\xi_r)\\
&={\bf P}_{n}(\xi_1,\xi_2,\dots,\xi_r) \quad (\text{by~\cite[Lemma 4.8]{S2018}})\\
&=\sum\limits^r_{i=1}\xi^n_i.\tag*{\qed}
\end{align*} \renewcommand{\qed}{}
\end{proof}

For $r=2$, we obtain the Fibonacci recurrence $\tilde t_2(n)=\tilde t_2(n-1)+\tilde t_2(n-2)$; however, with different initial conditions. Thus we obtain Lucas numbers and we have the following formula.
 \begin{Corollary}\label{3.10}
$\tilde t_2(n)=\big(\frac{1+\sqrt{5}}{2}\big)^n+\big(\frac{1-\sqrt{5}}{2}\big)^n$.
\end{Corollary}

\section[The number of support $\tau$-tilting modules over Nakayama algebras]{The number of support $\boldsymbol{\tau}$-tilting modules\\ over Nakayama algebras}\label{sect 4}

 In this section, we will apply our results to give a new proof of a theorem by Asai, see \cite[Theorem~4.1(1) and~(2)]{S2018}

 Applying Proposition~\ref{3.2}, we obtain the following recurrence relation for the number of support $\tau$-tilting modules over $\Lambda^r_n$. See Table~\ref{t2} in Section~\ref{sect 5} for explicit values of $s_r(n)$.

 \begin{Proposition}[{\cite[Theorem 4.1(1)]{S2018}}]\label{4.1}
\begin{gather*} s_r(n)=2s_r(n-1)+\sum^r_{i=2}C_{i-1}\cdot s_r(n-i).\end{gather*}
 \end{Proposition}

\begin{proof} Since $\Lambda=\Lambda^r_n$, we have $\Lambda_{<i}\cong \Lambda^r_{i-1}$ and $\Lambda_{>i}\cong \Lambda^r_{n-i}$. Thus Proposition \ref{3.2}(1) yields
\[
s_r(n)=\sum\limits^{n}_{i=1}t_r(i-1)\cdot s_r(n-i)+ t_r(n).
\]
Therefore,
\begin{gather*}
s_r(n)-s_r(n-1)
= \sum\limits^{n}_{i=1}t_r(i-1)\cdot s_r(n-i)+ t_r(n)\\
\hphantom{s_r(n)-s_r(n-1)=}{} - \left(\sum\limits^{n-1}_{i=1}t_r(i-1)\cdot s_r(n-1-i)+ t_r(n-1)\right)\\
\hphantom{s_r(n)-s_r(n-1)}{} =s_r(n-1)+ \sum\limits^{n-1}_{i=1}(t_r(i)-t_r(i-1))\cdot s_r(n-1-i)+(t_r(n)-t_r(n-1)).
\end{gather*}
By Lemma \ref{2.4}, for all $i$ we have
\[ t_r(i)=t_r(i-1) + \sum\limits^r_{\ell=2}C_{\ell-1}\cdot t_r(i-\ell).\] Using this equality in the identity above, we get
\begin{gather*}
s_r(n)-s_r(n-1)=
s_r(n-1)+\sum\limits^{n-1}_{i=1}\sum^r_{\ell=2}C_{\ell-1}\cdot t_r(i-\ell)\cdot s_r(n-1-i)\\
\hphantom{s_r(n)-s_r(n-1)=}{} +\sum^r_{\ell=2}C_{\ell-1}\cdot t_r(n-\ell)
\end{gather*}
and since $s_r(n-\ell-i)=0 $ whenever $i>n-\ell$, we obtain
\begin{gather*}
s_r(n)-s_r(n-1)=s_r(n-1)+ \sum^r_{\ell=2}C_{\ell-1}\cdot\sum^{n-\ell}_{i=1}t_r(i-1)\cdot s_r(n-\ell-i)\\
\hphantom{s_r(n)-s_r(n-1)=}{}
+\sum^r_{\ell=2}C_{\ell-1}\cdot t_r(n-\ell)\\
\hphantom{s_r(n)-s_r(n-1)}{}
=s_r(n-1)+ \sum^r_{\ell=2}C_{\ell-1}\cdot\left(\sum^{n-\ell}_{i=1}t_r(i-1)\cdot s_r(n-\ell-i)+ t_r(n-\ell)\right)\\
\hphantom{s_r(n)-s_r(n-1)}{}=s_r(n-1)+ \sum^r_{\ell=2}C_{\ell-1} \cdot s_r(n-\ell) \quad \text{by Proposition \ref{3.2}(1)}.
\end{gather*}
Hence, $s_r(n)=2s_r(n-1)+\sum\limits^r_{i=2}C_{i-1}\cdot s_r(n-i).$
\end{proof}

\begin{Remark}\label{4.2}
If $r\ge n$ then $\Lambda_n^r$ is a hereditary algebra of Dynkin type $\mathbb{A}_n$ and the support $\tau$-tilting modules are in bijection with the clusters in the corresponding cluster algebra. In particular, $s_r(n)=C_{n+1}$.
This fact also can be obtained directly from Proposition \ref{4.1} and the identity
$C_n=\sum\limits^n_{i=1}C_{i-1}\cdot C_{n-i}$ (see Remark \ref{2.5}). Indeed, the following equation
\begin{gather*}
s_n(n)-C_{n+1}= 2s_n({n-1})+\sum^n_{i=2}C_{i-1}\cdot s_n({n-i})-\left(\sum^{n+1}_{i=1}C_{i-1}\cdot C_{n+1-i}\right)\\
\hphantom{s_n(n)-C_{n+1}}{} = C_0\cdot s_n({n-1})+\sum^n_{i=2}C_{i-1}\cdot s_n({n-i})-\left(\sum^{n}_{i=1}C_{i-1}\cdot C_{n+1-i}\right)\\
\hphantom{s_n(n)-C_{n+1}=}{}+ s_n({n-1})-C_{n}\cdot C_0\\
\hphantom{s_n(n)-C_{n+1}}{} =\sum^{n}_{i=1}C_{i-1}\cdot (s_n({n-i})-C_{n+1-i})+(s_n({n-1})-C_{n}\cdot C_0)\\
\hphantom{s_n(n)-C_{n+1}}{} =\sum^{n}_{i=1}C_{i-1}\cdot (s_{n-i}({n-i})-C_{n+1-i})+(s_{n-1}({n-1})-C_{n})
\end{gather*}
implies $s_n(n)=C_{n+1}$ by induction.
\end{Remark}

Applying Theorem~\ref{3.8}, we obtain the following recurrence relation for the number of support $\tau$-tilting modules over $\widetilde{\Lambda}^r_n$. See Table \ref{t4} in Section~\ref{sect 5} for explicit values of $\widetilde{s}_r(n)$.

 \begin{Proposition}[{\cite[Theorem 4.1(2)]{S2018}}]\label{4.3}
\begin{gather*}\tilde s_r(n)=2 \tilde s_r(n-1)+\sum\limits^r_{i=2}C_{i-1}\cdot \tilde s_r(n-i).\end{gather*}
 \end{Proposition}

\begin{proof}We need to show
\[
\tilde s_r(n)-2\tilde s_r(n-1)-\sum\limits^r_{i=2}C_{i-1}\cdot \tilde s_r(n-i)=0.
\]
Using Proposition~\ref{3.5} and the equation $\tilde s_r(n)=\big|\pstilt\widetilde{\Lambda}^r_n\big|+\tilde t_r(n)$, we obtain
\begin{gather*}\tilde s_r(n)=\sum\limits^{n-1}_{\ell=1}\ell\cdot t_r(\ell-1)\cdot s_r(n-\ell-1)+n\cdot t_r(n-1)+\tilde t_r(n).\end{gather*}
 Therefore,
\begin{gather*}
\tilde s_r(n)-2\tilde s_r(n-1)-\sum^r_{i=2}C_{i-1}\cdot \tilde s_r(n-i)\\
{}=\sum\limits^{n-1}_{\ell=1}\ell\cdot t_r(\ell-1)\cdot s_r(n-\ell-1)+n\cdot t_r(n-1)+\tilde t_r(n)\\
\quad{} -2\left(\sum\limits^{n-2}_{\ell=1}\ell\cdot t_r(\ell-1)\cdot s_r(n-\ell-2)+(n-1)\cdot t_r(n-2)+\tilde t_r(n-1)\right)\\
\quad{} -\sum^r_{i=2}C_{i-1}\cdot \!\left(\sum\limits^{n-i-1}_{\ell=1}\ell\cdot t_r(\ell-1)\cdot s_r(n-i-\ell-1)+(n-i)\cdot t_r(n-i-1)+\tilde t_r(n-i)\right).
\end{gather*}
Note that every term in the summation
 $\sum\limits^r_{i=2}C_{i-1}\sum\limits^{n-2}_{\ell=n-i}\ell\cdot t_r(\ell-1)\cdot s_r(n-i-\ell-1)$ is zero, and therefore the whole expression is equal to
\begin{gather*}
= \sum\limits^{n-2}_{\ell=1}\ell\cdot t_r(\ell-1)\cdot\left(s_r(n-\ell-1)-2s_r(n-\ell-2)-\sum^r_{i=2}C_{r-1}\cdot s_r(n-\ell-1-i)\right)\\
\quad{}+n\cdot t_r(n-1)+\big((n-1)\cdot t_r(n-2)-2(n-1)\cdot t_r(n-2)\big)\\
\quad{}-\sum^r_{i=2}(n-i)\cdot C_{i-1}\cdot t_r(n-i-1)+\left(\tilde t_r(n)-2\tilde t_r(n-1)-\sum^r_{i=2}C_{i-1}\cdot\tilde t_r(n-i)\right).
\end{gather*}

Now, the parenthesis in the first sum is zero, by Proposition \ref{4.1}, the large parenthesis in the second row can be included in the first sum of the third row as the $i=1$ term, and the parenthesis in the third row is equal to $-\tilde t_r(n-1)$ by Theorem \ref{3.8}. So the whole expression is equal to
\begin{gather*}
=n\cdot t_r(n-1)-\sum^r_{i=1}(n-i)\cdot C_{i-1}\cdot t_r(n-i-1)-\tilde t_r(n-1)\\
=n\cdot \left(t_r(n-1)-\sum^r_{i=1}C_{i-1}\cdot t_r(n-i-1)\right)+\sum^r_{i=1}i\cdot C_{i-1}\cdot t_r(n-i-1)-\tilde t_r(n-1)\\
=\sum^r_{i=1}i\cdot C_{i-1}\cdot t_r(n-i-1)-\tilde t_r(n-1)\quad \text{(by~Lemma~\ref{2.4})}\\
=\tilde t_r(n-1)-\tilde t_r(n-1) \quad \text{(by~Proposition \ref{3.7})}\\
=0.\tag*{\qed}
\end{gather*}\renewcommand{\qed}{}
\end{proof}

\section{Examples}\label{sect 5}
In this section, we give examples of the numbers of (support) $\tau$-tilting modules over $\zL^r_n$ and $\widetilde{\Lambda}^r_n$ (see Tables \ref{t1}--\ref{t4}).

\begin{table}[th!]\centering
\caption{The number of $\tau$-tilting modules of $\zL^r_n$.}\label{t1}
\vspace{1mm}
\begin{tabular}{l|cccccccccccc}
 \hline
 \diagbox{$r$}{$t_r(n)$}{$n$} & 1 & 2 & 3&4&5&6&7&8&9&10&11&12\\
 \hline
1 & 1&1&1&1&1&1&1&1&1&1&1&1\\
2 &1& 2 & 3 &5&8&13&21&34&55 &89&144&233\\
3 &1& 2 & 5 &9&18&37&73&146&293&585&1170&2341\\
4 &1& 2 & 5 & 14&28&62&143&331&738&1665&3780&8576\\
5 &1&2 & 5 & 14&42&90&213&527& 1326&3317&8022&19608\\
6 &1&2 & 5 & 14&42&132&297&737& 1914&5081&13566&35862\\
 \hline \end{tabular}
\end{table}

\begin{table}[th!]\centering
\caption{The number of support $\tau$-tilting modules of $\zL^r_n$.}\label{t2}
\vspace{1mm}
\begin{tabular}{l|cccccccccccc}
 \hline
 \diagbox{$r$}{$s_r(n)$}{$n$} & 1 & 2 & 3&4&5&6&7&8&9&10&11&12\\
 \hline
1 & 2&4& 8&16&32&64&128&256&512&1024&2048&4096 \\
2 & 2 & 5 &12&29&70&169&408&985&2378&5741&13860&33461 \\
3 & 2 & 5 &14&37&98&261&694&1845&4906&13045&34686&92229\\
4 & 2 & 5 & 14&42&118&331&934&2645&7476&21120&59676&168649\\
5 &2 & 5 & 14&42&132&387&1130& 3317&9786&28932&85352&251613\\
6 &2 & 5 & 14&42&132&429&1298& 3905&11802&35862&109376&333933\\
 \hline \end{tabular}
\end{table}

 \begin{table}[th!]\centering
\caption{The number of $\tau$-tilting modules of $\widetilde{\Lambda}^r_n$.}\label{t3}
\vspace{1mm}
\begin{tabular}{l|cccccccccccc}
 \hline
 \diagbox{$r$}{$\tilde t_r(n)\vphantom{\Big|}$}{$n$} & 1 & 2 & 3&4&5&6&7&8&9&10&11&12\\
 \hline
1 & 1&1& 1&1&1&1&1&1&1&1&1&1\\
2 & 1 & 3 &4&7&11&18&29&47&76&123&199&322\\
3 & 1 & 3 &10&15&31&66&127&255&514&1023&2047&4098\\
4 & 1 & 3 & 10&35&56&126&302&715&1549&3498&7897&18158\\
5 &1 & 3 & 10&35&126&210&498&1275&3313&8398&19691&48062 \\
6 &1 & 3 & 10&35&126&462&792&1947 &5203&14278&39095&104006\\
\hline \end{tabular}
\end{table}
\begin{table}[th!]\centering
\caption{The number of support $\tau$-tilting modules of $\widetilde{\Lambda}^r_n$.}\label{t4}
\vspace{1mm}
\begin{tabular}{l|cccccccccccc}
 \hline
 \diagbox{$r$}{$\tilde s_r(n)$}{$n$} & 1 & 2 & 3&4&5&6&7&8&9&10&11&12\\
 \hline
1 & 2&4& 8&16&32&64&128&256& 512&1024&2048&4096\\
2 & 2 & 6 &14&34&82&198&478&1154&2786&6726&16238&39202 \\
3 & 2 & 6 &20&50&132&354&940&2498&6644&17666&46972&124898\\
4 & 2 & 6 & 20&70&182&504&1430&4078&11504&32466&91742&259348\\
5 &2 & 6 & 20&70&252&672&1920&5646&16796&49966&147028&432724 \\
6 &2 & 6 & 20&70&252&924&2508&7326&22088&67606&208012&638356 \\
 \hline \end{tabular}
\end{table}

\subsection*{Acknowledgements}

The first author was partially supported by NSFC (Grant No.~11971225). The second author was supported by the NSF grant DMS-1800860 and by the University of Connecticut. The authors also thank the referees for the useful and detailed suggestions.

\pdfbookmark[1]{References}{ref}
\LastPageEnding

\end{document}